\documentclass{amsart}
\usepackage{amsmath}
\usepackage{amssymb}
\def\T{\mathbb{T}}
\def\dom{\mathop{\mathrm{dom}}\nolimits}

\def\bd{\mathrm{bd}}
\def\cl{\mathrm{cl}}
\def\int{\mathrm{int}}

\def\GG{\mathcal{G}}
\def\HH{\mathcal{H}}
\def\SSC{\mathcal{S}}

\def\XD{X^{\sharp}}

\theoremstyle{plain}
\newtheorem{theorem}{Theorem}
\newtheorem{lemma}{Lemma}

\theoremstyle{definition}

\title{ Every Separable Metrizable Space has a Proper Dyadic Subbase}
\author{Haruto Ohta}
\address[H. Ohta]{Faculty of Education, Shizuoka University}

\author{Hideki Tsuiki}
\address[H. Tsuiki]{Graduate School of Human and Environmental Studies, Kyoto University}

\author{Kohzo Yamada}
\address[K. Yamada]{Faculty of Education, Shizuoka University}

\begin{document}

\begin{abstract}
The notions of a proper dyadic subbase and an independent subbase was introduced by H. Tsuiki to investigate
in $\{0, 1,\bot\}^\infty$-codings of topological spaces.
We show that every separable metrizable space $X$ has a proper dyadic subbase whose restriction to the perfect set $\XD$ defined by the Cantor-Bendixson theorem forms an independent subbase of $\XD$.
\end{abstract}

\maketitle

\section{Introduction}

The notion of a proper dyadic subbase was introduced  in \cite{T}
to investigate in $\{0, 1,\bot\}^\infty$-codings of topological spaces, and 
domain representations induced by a proper dyadic subbase are investigated
in \cite{TT}.
It is shown in \cite{OTY} that every dense-in-itself separable metrizable space
has an independent subbase, which is a proper dyadic subbase with some
additional properties.
In this article, we generalize this result and show that 
every separable metrizable space $X$ has a proper dyadic subbase $S$ whose
restriction to the perfect set $\XD$ defined by the Cantor-Bendixson theorem
forms an independent subbase of $\XD$ when $\XD \ne \emptyset$.

It is also proved in \cite{OTY}  that if $X$ is a dense-in-itself separable metrizable
space with $\dim X = m$, then $X$ has an independent subbase $\SSC$ with 
$\deg\, \SSC = m$.  We also generalize this result and show that 
if $X$ is a separable metrizable space with $\dim X = m$, then $X$ has a proper dyadic
subbase $\SSC$ with the above property such that $\deg\, \SSC = m$.  

\medbreak

\noindent
{\bf Definitions:}

A {\it dyadic subbase} of a space $X$ is a subbase
$\SSC = \{S_{n,i}: n < \omega, i <2\}$ indexed with $\omega \times 2$ 
which is composed of regular open sets and
$S_{n,1}$ is the exterior of $S_{n,0}$ for each $n < \omega$.
Let $\T = \{0,1,\bot\}$, $\dom(\sigma) = \{n : \sigma(n) \in \{0,1\}\}$ for $\sigma \in \T^{\omega}$, and $\T^{*}  = \{\sigma\in \T^{\omega} : |\dom(\sigma)| < \infty\}$.
For a dyadic subbase $\SSC= \{S_{n,i}: n < \omega, i <2\}$ and $\sigma \in \T^{\omega}$, let 
\begin{align*}
&S(\sigma)=\bigcap_{k \in \dom(\sigma)}S_{k, \sigma(k)}\quad
\text{and}\quad 
\bar{S}(\sigma)=\bigcap_{k \in \dom(\sigma)} {\cl\, S_{k,\sigma(k)}}
\end{align*}
denote the corresponding subsets of $X$.
The family $\{S(\sigma): \sigma \in \T^{*}\}$ is the base generated by $\SSC$.
We say that a dyadic subbase $\SSC = \{S_{n,i}: n < \omega, i <2\}$ is {\em proper} if 
$\cl\, S(\sigma)=\bar{S}(\sigma)$ for every $\sigma \in \T^*$.  

An {\em independent subbase} is a dyadic subbase 
$\SSC = \{S_{n,i}: n < \omega, i <2\}$ such that $S(\sigma)$ is not empty
for every $\sigma \in \T^*$.  An independent subbase is always proper \cite{OTY}.

For a dyadic subbase $\SSC=\{S_{n,i}:n<\omega, i<2\}$ of 
a separable metric space $X$ and $x \in X$, we define
$\deg_{\mathcal{S}}(x) = |\{n<\omega :  x \in X\setminus(S_{n,0}\cup S_{n,1})\}|$ and
$\deg \SSC = \sup \{\deg_{\mathcal{S}}(x) : x \in X \}$.

\section{Topological Lemmas}

A space $X$ is called {\it hereditarily paracompact} if 
every subspace of $X$ is paracompact.
%Lemmas 1 and 2 below are due to K. Yamada.

\begin{lemma}
Let $X$ be a hereditarily paracompact, Hausdorff space
and $Y$ a closed subspace of $X$.
For every open sets $U_0$ and $U_1$ in $Y$ such that
$U_0\cap U_1=\emptyset$, there exist
open sets $V_0$ and $V_1$ in $X$ such that $V_i\cap Y=U_i$
for each $i=0,1$ and 
$\cl_XV_0\cap\cl_XV_1\subseteq Y$.
\end{lemma}
\begin{proof}
By \cite[Theorem 2.1.7]{E}, there exist open sets $W_0$ and $W_1$ in
$X$ such that $W_i\cap Y=U_i$ for each $i=0,1$ and
$W_0\cap W_1=\emptyset$.
Put
$$
Z=W_0\cup W_1\cup(X\setminus Y).
$$
For each $i=0,1$ and each $x\in U_i$, 
since every paracompact Hausdorff space is regular,
there is an open neighborhood $G_x$ of $x$ in $X$ such that
$\cl_XG_x\subseteq W_i$.
Then the collection
$\GG=\{G_x:x\in U_0\cup U_1\}\cup\{X\setminus Y\}$ is an 
open cover of $Z$.
Since $Z$ is paracompact, 
$\GG$ has a locally finite (in $Z$) open refinement $\HH$.
Define 
$$
V_i=\bigcup\{H\in\HH:H\cap U_i\neq\emptyset\}
$$ 
for each $i=0,1$. 
Then, $V_i$ is open in $X$ since $Z$ is open in $X$,
and $V_i\cap Y=U_i$ for each
$i=0,1$.
If $H\in\HH$ and $H\cap U_i\neq\emptyset$, then 
$H\subseteq G_x$ for some $x\in U_i$, which implies that
$\cl_ZH\subseteq\cl_XH\subseteq\cl_XG_x\subseteq W_i$.
By this fact and local finiteness of $\HH$ in $Z$, 
$$
\cl_XV_i\cap Z=\cl_ZV_i
=\bigcup\{\cl_ZH:H\in\HH, H\cap U_i\neq\emptyset\}\subseteq W_i
$$
for each $i=0,1$.
Hence, $\cl_XV_0\cap\cl_XV_1\cap Z\subseteq W_0\cap W_1=\emptyset$,
which implies that 
$\cl_XV_0\cap\cl_XV_1\subseteq Y$.
\end{proof}

Let $X$ be a separable metric space.
By the Cantor-Bendixson theorem, there exists a perfect 
($=$ dense in itself) closed set $\XD$ in $X$ 
such that $X\setminus \XD$ is at most countable.
Note that $\dim(X\setminus \XD)=0$.
By a {\it clopen set} we mean an open and closed set.
We now call a subset $G$ of $X$ a {\it half-clopen set}
in $X$ if $G$ is regular open in $X$ and 
$\bd_XG\subseteq \XD$.

\begin{lemma}
Let $X$ be a separable metric space, $U$ a regular open set in $\XD$
and $W$ an open set in $X$ with $\cl_{\XD}U\subseteq W$.
Then, there exists a half-clopen set $V$ in $X$ such that
$V\cap \XD=U$ and $V\subseteq W$.
\end{lemma}
\begin{proof}
By Lemma 1, there exist open sets $V_0$ and $V_1$ in $X$
such that 
$$
V_0\cap \XD=U,\ V_1\cap \XD=\XD\setminus\cl_{\XD}U\  
\text{ and }\ \cl_XV_0\cap\cl_XV_1\subseteq \XD.
$$
Since $X\setminus \XD$ is 0-dimensional, 
there exists a clopen set $H_0$ in $X\setminus \XD$ such that 
$\cl_XV_0\cap(X\setminus \XD)\subseteq H_0$ and
$H_0\cap\cl_XV_1=\emptyset$.
Observe that $V_0\cup H_0$ is a half-clopen set in $X$.
Put $F=H_0\setminus W$.
Since $F$ is closed in $X$ and $F\cap \XD=\emptyset$,
there is an open set $G$ in $X$ such that $F\subseteq G$
and $\cl_XG\cap \XD=\emptyset$.
Since $X\setminus \XD$ is 0-dimentional, 
there is a clopen set $H_1$ in $X\setminus \XD$ with 
$F\subseteq H_1\subseteq G$.
Then, $H_1$ is clopen in $X$.
Finally, we obtain a required half-clopen set $V$ 
in $X$ by letting $V=(V_0\cup H_0)\setminus H_1$.
\end{proof}

\section{Main Theorems}

\begin{theorem}
Let $X$ be a separable metric space.
There exists a proper dyadic subbase $\SSC$ 
consisting of half-clopen sets in $X$ such that
$\{S\cap \XD:S\in\SSC\}$ is an independent subbase of $\XD$
when $\XD\neq\emptyset$.
\end{theorem}
\begin{proof}
If $\XD=\emptyset$, then for every countable base 
$\{B_n:n<\omega\}$ of $X$ consisting of clopen sets,
$\{B_n,X\setminus B_n:n<\omega\}$ is a required proper 
dyadic subbase of $X$.
So, we assume that $\XD\neq\emptyset$.
In \cite{OTY} we proved that every dense in itself, separable
metric space has an independent subbase.
By applying it to $\XD$, we can define an independent subbase
$\SSC=\{S_{n,i}:n<\omega, i<2\}$ of $\XD$.
The proof begins by taking a collection
$\{U_{n,i}:n<\omega,i<2\}$ of non-empty regular open sets 
in $\XD$ such that $\cl_{\XD}U_{n,0}\subseteq U_{n,1}$
for each $n<\omega$ and, for each $x\in \XD$ and each
neighborhood $G$ of $x$ in $\XD$, 
there is $n<\omega$ such that
$x\in U_{n,0}$ and $U_{n,1}\subseteq G$.
Now, we can assume, in addition, that
for each $n<\omega$ there is an open set $U_{n,1}^\ast$ in
$X$ with $U_{n,1}^\ast\cap \XD=U_{n,1}$ 
and for each $x\in \XD$ and each neighborhood $G$ of $x$ in $X$,
there is $n<\omega$ such that
$x\in U_{n,0}$ and $U_{n,1}^\ast\subseteq G$.
Then, the independent subbase $\SSC$ of $\XD$ is  
defined by induction on $n<\omega$ as follows.
At the first step, $S_{0,0}$ and $S_{0,1}$ are defined to be
$U_{0,0}$ and $X\setminus\cl_{\XD}U_{0,0}$, respectively.
At the $n$-th step, $S_{n,0}$ and $S_{n,1}$ are defined
so as to satisfy that
\begin{gather}
S_{n,1}=\XD\setminus\cl_{\XD}S_{n,0}, 
\label{eq:2-1a}\\
(\forall\sigma\in{^{n+1}2})\left(
\cl_{\XD}\bigcap_{k\leq n}S_{k,\sigma(k)}=
\bigcap_{k\leq n}\cl_{\XD}S_{k,\sigma(k)}\right), 
\label{eq:2-1b}\\
(\forall\sigma\in{^{n+1}2})\left(
\bigcap_{k\leq n}S_{k,\sigma(k)}\neq\emptyset\right),\textrm{ and }
\label{eq:2-1c}\\
(\forall x\in U_{n,0})(\exists\,\Gamma\subseteq{n+1})(\exists 
\sigma\in{^\Gamma 2})
\left(x\in\bigcap_{k\in\Gamma}S_{k,\sigma(k)}\subseteq U_{n,1}\right).
\label{eq:2-1d}
\end{gather}
If such $S_{n,0}$ and $S_{n,1}$ were defined for all $n<\omega$,
then it follows from (1)--(4) that the resulting collection
$\SSC=\{S_{n,i}:n<\omega,i<2\}$ is an independent subbase of $\XD$.
After having defined $\SSC$, we extend, 
by induction on $n<\omega$ again, each $S_{n,i}\in\SSC$ to 
a half-clopen set $S_{n,i}^\ast$ in $X$ satisfying that
\begin{gather}
S_{n,1}^\ast=X\setminus\cl_XS_{n,0}^\ast, 
\label{eq:2-2a}\\
(\forall i<2)\left(S_{n,i}^\ast\cap \XD=S_{n,i}\right), 
\text{ and}
\label{eq:2-2b}\\
(\forall x\in U_{n,0})(\exists\,\Gamma\subseteq{n+1})
(\exists \sigma\in{^\Gamma 2})
\left(x\in\bigcap_{k\in\Gamma}S_{k,\sigma(k)}^\ast
\subseteq U_{n,1}^\ast\right).
\label{eq:2-2c}
\end{gather}
By Lemma 2, there is a half-clopen set $S_{0,0}^\ast$ in
$X$ such that $S_{0,0}^\ast\cap \XD=S_{0,0}$ and
$S_{0,0}^\ast\subseteq U_{0,1}^\ast$.
Put $S_{0,1}^\ast=X\setminus\cl_XS_{0,0}^\ast$.
Then $S_{0,1}^\ast$ is also a half-clopen set in $X$ with
$S_{0,1}^\ast\cap \XD=S_{0,1}$.
Let $n<\omega$ and assume that a half-clopen set $S_{k,i}^\ast$
has been defined for each $k<n$ and each $i<2$.
For each $\sigma\in{^n}2$, let
$$
S(\sigma)=\bigcap_{k<n}S_{k,\sigma(k)}\quad
\text{and}\quad 
S^\ast(\sigma)=\bigcap_{k<n}S_{k,\sigma(k)}^\ast.
$$
Then $S^\ast(\sigma)$ is a half-clopen set in $X$ with
$S^\ast(\sigma)\cap \XD=S(\sigma)$ by (\ref{eq:2-1b}).
At the $n$-th step in the proof of \cite[Theorem 1]{OTY},
the sets $S_{n,i}$, $i<2$, are defined as 
\begin{align*}
S_{n,0} &=\left(V_n\setminus\bigcup_{\sigma\in A}\cl_{\XD}G(\sigma)\right)\cup
\bigcup_{\sigma\in B}G(\sigma)\textrm{ and } \\
S_{n,1} &=\left((\XD\setminus\cl_{\XD}V_n)\setminus
\bigcup_{\sigma\in B}\cl_{\XD}G(\sigma)\right)\cup
\bigcup_{\sigma\in A}G(\sigma),
\end{align*}
where $V_n$ is a carefully chosen (using \cite[Lemma 10]{OTY}) regular 
open set in $\XD$ with 
$\cl_{\XD}U_{n,0}\subseteq V_n\subseteq \cl_{\XD}V_n\subseteq U_{n,1}$,
$$
A=\{\sigma\in{^n2}:S(\sigma)\subseteq V_n\},\quad
B=\{\sigma\in{^n2}:S(\sigma)\cap\cl_{\XD}V_n=\emptyset\},
$$ and
for each $\sigma\in A\cup B$, $G(\sigma)$ is a non-empty regular open
set in $\XD$ such that
$\cl_XG(\sigma)\subseteq S(\sigma)$ and
$S(\sigma)\setminus\cl_XG(\sigma)\neq\emptyset$.
By Lemma 2, there is a half-clopen set $V_n^\ast$ such that
$$
V_n^\ast\cap \XD=V_n\quad\text{and}\quad V_n^\ast\subseteq U_{n,1}^\ast.
$$
Moreover, for each $\sigma\in A$, there is a half-clopen set
$G^\ast(\sigma)$ such that $G^\ast(\sigma)\cap \XD=G(\sigma)$
and $G^\ast(\sigma)\subseteq V_n^\ast\cap S^\ast(\sigma)$,
and, for each $\sigma\in B$,
there is a half-clopen set $G^\ast(\sigma)$ 
such that $G^\ast(\sigma)\cap \XD=G(\sigma)$ and 
$G^\ast(\sigma)\subseteq S^\ast(\sigma)\setminus\cl_XV_n^\ast$.
Define
\begin{align*}
S_{n,0}^\ast &=\left(V_n^\ast\setminus
\bigcup_{\sigma\in A}\cl_XG^\ast(\sigma)\right)\cup
\bigcup_{\sigma\in B}G^\ast(\sigma)\textrm{ and } \\
S_{n,1}^\ast &=\left((X\setminus\cl_XV_n^\ast)\setminus
\bigcup_{\sigma\in B}\cl_XG^\ast(\sigma)\right)\cup
\bigcup_{\sigma\in A}G^\ast(\sigma).
\end{align*}
Then, $S_{n,i}^\ast$, $i<2$, are half-clopen sets and satify
(\ref{eq:2-2a}) and (\ref{eq:2-2b}).
By a similar argument to the proof in \cite{OTY} that
$S_{n,0}$ and $S_{n,1}$ satisfy (\ref{eq:2-1d}),
it can be checked that $S_{n,0}^\ast$ and $S_{n,1}^\ast$
satisfy (\ref{eq:2-2c}).
Hence, the induction is complete.
It follows from the additional assumption on the sets 
$U_{n,1}^\ast$ and (\ref{eq:2-2c}),
for each $x\in \XD$ and each neighborhood $G$ of $x$ in $X$,
there is a finite set $\Gamma\subseteq\omega$ and
$\sigma\in{^\Gamma}2$ such that
$$
x\in\bigcap_{k\in\Gamma}S_{k,\sigma(k)}^\ast\subseteq G.
$$
Finally, since $X\setminus \XD$ is $0$-dimensional
and $\XD$ is closed in $X$, 
there is a countable base
$\{H_n:n<\omega\}$ of $X\setminus \XD$ consisting of clopen sets
in $X$.
Then, 
$$
\SSC^\ast=\{S_{n,0}^\ast,S_{n,1}^\ast:n<\omega,i<2\}\cup
\{H_n,X\setminus H_n:n<\omega\}
$$ is a required proper
dyadic subbase of $X$.
\end{proof}

\par\bigskip
%For a dyadic subbase $\SSC=\{S_{n,i}:n<\omega, i<2\}$ of 
%a separable metric space $X$, we define
%$\ord\,\SSC=\ord\{X\setminus(S_{n,0}\cup S_{n,1}):n<\omega\}$.

\begin{theorem}
Let $X$ be a separable metric space with $\dim\,X=m$.
Then, $X$ has a proper dyadic subbase $\SSC$ with $\deg\,\SSC=m$.
\end{theorem}
\begin{proof}
If $\XD=\emptyset$,
then for every countable base 
$\{B_n:n<\omega\}$ of $X$ consisting of clopen sets,
$\SSC=\{B_n,X\setminus B_n:n<\omega\}$ is a proper 
dyadic subbase of $X$ with $\deg\,\SSC=0$.
So, we consider the case $\XD\neq\emptyset$.
Then, by \cite[Corollary 4]{OTY}, there exists an independent
subbase $\SSC$ of $\XD$ with $\deg\,\SSC=m$.
By the same argument as in the proof of Theorem 1, we can
extend $\SSC$ to a proper dyadic subbase $\SSC^\ast$ of $X$.
Then, $\deg\,\SSC^\ast=\deg\,\SSC=m$, since
each member of $\SSC$ is either a half-clopen set
or a clopen set in $X$. 
\end{proof}

\bibliographystyle{abbrvnat}

\end{document}